\documentclass[english]{smfart}
\usepackage[T1]{fontenc}
\usepackage[utf8]{inputenc}
\usepackage{amssymb,url,upref,xspace,smfenum, mathrsfs}
\RequirePackage[pdftex,
pdfview=XYZ,
pdfstartview=FitV,
pdffitwindow,
hyperfootnotes=false,
raiselinks=true,
hypertexnames=false,
colorlinks=true,
bookmarks=true,
bookmarksopen=false,
anchorcolor=black,
citecolor=black,
linkcolor=black,
filecolor=black,
urlcolor=black]{hyperref}

\usepackage{tikz}
\usetikzlibrary{decorations.markings}
\usetikzlibrary{calc}
\usetikzlibrary{through}
\usetikzlibrary{backgrounds}
\newcommand{\crosss}{+(-1.4142pt,-1.4142pt) -- +(1.4142pt,1.4142pt) +(1.4142pt,-1.4142pt) -- +(-1.4142pt,1.4142pt) +(1.4142pt,-1.4142pt)}

\newcommand{\CC}{\mathbb{C}}
\newcommand{\DD}{\mathbb{D}}
\newcommand{\PP}{\mathbb{P}}
\newcommand{\ZZ}{\mathbb{Z}}
\newcommand{\cD}{\mathscr{D}}
\newcommand{\cG}{\mathscr{G}}
\newcommand{\cO}{\mathscr{O}}

\newcommand\loccit{loc.\kern3pt cit.{}\xspace}
\newcommand\cf{see\kern.3em}
\newcommand\eg{e.g.\kern.3em}
\newcommand\ie{i.e.,\ }
\let\moins\smallsetminus

\let\geq\geqslant
\let\wh\widehat
\let\wt\widetilde

\let\la\lambda
\let\epsilon\varepsilon

\newcommand{\lcr}{[\![}
\newcommand{\rcr}{]\!]}
\newcommand{\lpr}{(\!(}
\newcommand{\rpr}{)\!)}
\newcommand{\bbullet}{{\scriptscriptstyle\bullet}}
\newcommand{\cbbullet}{{\raisebox{1pt}{$\bbullet$}}}
\newcommand{\image}{\mathrm{im}}
\newcommand{\gr}{\mathrm{gr}}
\newcommand{\nb}{\mathrm{nb}}
\newcommand{\rc}{\mathrm{c}}
\newcommand{\rv}{\mathrm{v}}
\newcommand{\rd}{\mathrm{d}}
\newcommand{\rT}{\mathrm{T}}
\newcommand{\res}{\mathrm{res}}
\newcommand{\Cla}{\CC[\la]\langle\partial_\la\rangle}
\newcommand{\Czeta}{\CC[\zeta]\langle\partial_\zeta\rangle}
\newcommand{\Ozeta}{\cO_{\DD^c}[\zeta]\langle\partial_\zeta\rangle}
\newcommand{\Dzeta}{\cD_{\DD^c}[\zeta]\langle\partial_\zeta\rangle}
\newcommand{\FM}{{}^{\scriptscriptstyle\mathrm{L}\!}M}
\newcommand{\sfi}{\mathsf{i}}

\DeclareMathOperator{\diag}{diag}
\DeclareMathOperator{\id}{Id}
\DeclareMathOperator{\rk}{rk}
\DeclareMathOperator{\twopii}{2\pi\sfi}

\renewcommand\to{\mathchoice{\longrightarrow}{\rightarrow}{\rightarrow}{\rightarrow}}
\newcommand\mto{\mathchoice{\longmapsto}{\mapsto}{\mapsto}{\mapsto}}
\newcommand\hto{\mathrel{\lhook\joinrel\to}}
\newcommand{\isom}{\stackrel{\sim}{\longrightarrow}}

\let\oldbigoplus\bigoplus
\renewcommand{\bigoplus}{\mathop{\textstyle\oldbigoplus}\displaylimits}
\let\oldbigcup\bigcup
\renewcommand{\bigcup}{\mathop{\textstyle\oldbigcup}\displaylimits}
\let\oldcoprod\coprod
\renewcommand{\coprod}{\mathop{\textstyle\oldcoprod}\displaylimits}

\AtBeginDocument{%
\def\subjclassname
{\textup{2020} Mathematics Subject Classification}%
}

\makeatletter

\DeclareRobustCommand{\@pointrait}{\unskip\@addpunct{.}%
\mbox{}\ignorespaces}

\def\th@plain{%
\let\thm@indent\noindent
\thm@headfont{\fontfamily{ptm}\bfseries\itshape}%
\thm@notefont{\fontfamily{ptm}\bfseries\upshape}%
\thm@preskip.5\linespacing \@plus.5\linespacing
\thm@postskip\thm@preskip
\thm@headpunct{\bfseries\itshape\pointrait}
\itshape }
\makeatother

\let\oldtheequation\theequation
\def\numstareq{\let\oldtheequation\theequation\renewcommand{\theequation}{\oldtheequation$\,*$}}
\newenvironment*{starequation}{\numstareq\begin{equation*}\tag{\theequation}}{\end{equation*}\let\theequation\oldtheequation\ignorespaces}

\renewcommand{\thesubsection}{\thesection.\alph{subsection}}
\numberwithin{equation}{section}
\newtheorem{theoreme}[equation]{Theorem}
\newtheorem{lemme}[equation]{Lemma}
\newtheorem{proposition}[equation]{Proposition}
\newtheorem{corollaire}[equation]{Corollary}
\newtheorem*{criterion}{Criterion}
\theoremstyle{definition}
\newtheorem{remarque}[equation]{Remark}

\begin{document}
\frontmatter
\title[A short proof of a theorem of Cotti, Dubrovin and Guzzetti]{A short proof of a theorem of\\ Cotti, Dubrovin and Guzzetti}

\author[C.~Sabbah]{Claude Sabbah}
\address[C.~Sabbah]{CMLS, CNRS, École polytechnique, Institut Polytechnique de Paris, 91128 Palaiseau cedex, France}
\email{Claude.Sabbah@polytechnique.edu}
\urladdr{https://perso.pages.math.cnrs.fr/users/claude.sabbah/}

\subjclass{34M40, 32C38, 35A27}
\keywords{Laplace transformation, Stokes matrix, middle extension}

\begin{abstract}
We give a short proof of a theorem of G.\,Cotti, B.\,Dubrovin and D.\,Guzzetti (\cite{C-D-G18} and \cite{Guzzetti21}) asserting the vanishing of some entries of the Stokes matrices at coalescing points of an isomonodromic deformation.
\end{abstract}

\maketitle
{\let\\\relax\tableofcontents}
\mainmatter

\section*{Introduction}
A recent work \cite{Guzzetti21} of D.\,Guzzetti gives another approach to a result of G.\,Cotti, B.\,Dubrovin and D.\,Guzzetti in \cite{C-D-G18} that asserts the vanishing of some entries of the Stokes matrices at coalescing points of an isomonodromic deformation. We~recall the precise context in Section \ref{sec:shortproof}. The idea of D.\,Guzzetti is to exploit the property that the isomonodromic deformation in question can be obtained by Fourier-Laplace transformation from an isomonodromic deformation of a differential system with regular singularities.

Starting from this idea, we show how to recover the vanishing result by using the structure of such Stokes matrices as explained by Malgrange in \cite[Chap.\,XII]{Malgrange91} and proved in a topological way in \cite{D-H-M-S17}. We emphasize the property of intermediate (also called minimal, or~middle) extension of the differential system with regular singularity involved. Let us summarize our approach.
\begin{itemize}
\item
Given a differential system on the complex line with poles at a finite set, and having regular singularities including at infinity, we associate its minimal extension that we regard as a regular holonomic module on the Weyl algebra $\Cla$. Its~Laplace transform is a holonomic module on the complex line, having an irregular singularity at infinity. The result of Malgrange mentioned above enables one to compute representatives of Stokes matrices of the latter in terms of monodromies at finite distance of the former, and to show that the vanishing of some blocks of the Stokes matrices occurs when some relation between monodromies occurs. This general result is explained in Section \ref{sec:criterionvanishing}.
\item
This relation between monodromies has a dynamical interpretation when a universal isomonodromic deformation exists for the differential system with regular singularities. The parameter space is an open subset of the space of tuples of pairwise distinct singularities. Due to the simple geometry of the deformation space, one can interpret the relation between monodromies considered in the first point as the constancy of certain vanishing cycle sheaves on the singular set of the deformed family.

\item
A stationary phase result with parameter, as shown in \cite{D-S02a} (\cf Appendix \ref{app:B}), relates the formal decomposition at infinity of the Laplace transform of the regular holonomic module considered above with the vanishing cycle sheaves of the latter.

\item
These results lead to the following criterion, that we will only develop in the setting of Section \ref{sec:shortproof}. Let $U$ be a \emph{simply connected} open set in $\CC^n$ with coordinates $u_1,\dots,u_n$, and let $U^o$ be the complement of the diagonals $u_i=\nobreak u_j$ ($i\neq j\in\{1,\dots,n\}$). Let $X=U\times V$ be a neighborhood of $U\times\{z=0\}$ in $U\times\CC$ and set $X^o=U^o\times V$. Let $R_i$, for $i=1,\dots,n$, be a locally free $\cO_X$\nobreakdash-module with a flat logarithmic connection $\nabla_i$ having poles along~$U\times\{z=0\}$.

\begin{criterion}
Let $\cG$ be a locally free $\cO_X(*U)$-module with a flat connection. Assume that
\begin{enumerate}
\item
the formalization $(\wh G,\wh\nabla)=\cG_{\wh U}:=\cO_U\lpr z\rpr\otimes_{\cO_X}\cG$ decomposes as
\[
\bigoplus_{i=1}^n(R_i[z^{-1}],\nabla_i+\rd(u_i/z)),
\]
\item
There exists $u^o\!\in\!U^o$ such that $\cG_{|\{u^o\}\!\times\!V}$ is the the restriction to \hbox{$\{u^o\}\times V$} of the partial Laplace transform of a regular holonomic module on~$\CC_\lambda$ which is a minimal extension at each of its singularities at finite distance.
\end{enumerate}
If there exists a ``coalescing'' point $u^\rc$ in $U$ whose coordinates $u^\rc_i$ and $u^\rc_j$ coincide for some $i\neq j\in\{1,\dots,n\}$, then for each $u^o\in U^o$, the $(i,j)$- and $(j,i)$-entries of the Stokes matrices of $\cG_{|u=u^o}$ vanish.
\end{criterion}

The constancy of the vanishing cycle sheaves mentioned in the previous point is a consequence, via \cite{D-S02a}, of the constancy, for each $i$, of the local system on~$U$ attached to the flat residual connection of $(R_i,\nabla_i)$, as follows from the simple connectedness of $U$.

\item
In the setting of Section \ref{sec:shortproof}, according to \hbox{\cite[Prop.\,19.3]{C-D-G18}} (\cf Appendix \ref{app:A}), the first condition is satisfied. That the second one is also satisfied is proved by using the specific form of the connection matrix (\cf Lemma \ref{lem:point2}).
\end{itemize}

\section{A criterion for the vanishing of some entries of the Stokes matrices}\label{sec:criterionvanishing}

Let~$\CC_\la$ denote the complex line with coordinate $\la$ and let~\hbox{$u^o\!=\!(u_1^o,\dots,u_n^o)\!\in\!\CC^n$} having pairwise distinct coordinates. We denote by $\{u^o\}=\{u_1^o,\dots,u_n^o\}$ the corresponding subset of $\CC_\la$. Let~$L^o$ be a locally constant sheaf of finite rank $r$ on~$\CC_\la\moins \{u^o\}$. We will be mainly concerned with the perverse sheaf $j_*L^o$,\footnote{One usually shifts this sheaf by one, but we do not for the sake of simplicity.} where we denote by \hbox{$j:\CC_\la\moins\{u^o\}\hto\CC_\la$} the open inclusion. This sheaf is called the \emph{intermediate exten\-sion} of $L^o$ with respect to~$j$.

Let $(V^o,\nabla^o)$ be the meromorphic flat bundle on the affine line $\CC_\la$ with poles at~$u^o_i$ ($i=1,\dots,n$) and regular singularities there and at $\la=\infty$, whose sheaf of horizontal sections on $\CC_\la\moins u^o$ is the local system $L^o$. We can regard it as a meromorphic bundle on $\PP^1$ with poles at $u_1^o,\dots,u_n^o$ and $\infty$, equipped with a connection $\nabla^o$ having regular singularities there. We can also regard it as a regular holonomic (left) module over the Weyl algebra $\Cla$.

To $j_*L^o$ is associated, by the Riemann-Hilbert correspondence, a regular holonomic $\Cla$-module $(M^o,\nabla^o)$, that is called the \emph{middle extension} of $(V^o,\nabla^o)$.

Assume we are given a Fuchsian system with poles at $\{u^o\}$, and matrix\vspace*{-3pt}
\begin{equation}\label{eq:Fuchs}
B^o=\sum_{i=1}^n\frac{B_i^o}{\la-u_i^o}\,\rd\la.
\end{equation}
We regard it as the free $\CC[\la]$-module $E^o=\CC[\la]^r$ equipped with a logarithmic connection $\nabla^o$ whose matrix is given by the formula above. The associated meromorphic bundle with connection $(V^o,\nabla^o)$ is obtained by tensoring $E^o$ with the ring of rational functions $\CC[\la,((\la-u_i^o)^{-1})_{i=1,\dots,n}]$ having poles at most at $\{u^o\}$, that is,
\begin{equation}\label{eq:Vo}
V^o=\CC[\la,((\la-u_i^o)^{-1})_{i=1,\dots,n}]\otimes_{\CC[\la]}E^o,
\end{equation}
with connection naturally induced by $\nabla^o$ on $E^o$. We regard $(V^o,\nabla^o)$ as a left module over the Weyl algebra $\Cla$, and $(E^o,\nabla^o)$ is called a logarithmic lattice of~it. Although $(E^o,\nabla^o)$ is not a $\Cla$-module, it generates a $\Cla$-submodule $(M^o,\nabla^o)$ of $(V^o,\nabla^o)$ by setting\vspace*{-3pt}
\[
M^o=\sum_k(\nabla^o_{\partial_\la})^kE^o\subset V^o.
\]

The following lemma is standard.

\begin{lemme}\label{lem:middleext}
Assume that, for each $i=1,\dots,n$, the integral eigenvalues of $B_i^o$ are nonnegative. Then $(M^o,\nabla^o)$ is the middle extension of $(V^o,\nabla^o)$ at each of its poles.\qed
\end{lemme}

We consider the Laplace (also called Fourier) transformation with kernel $e^{\la\zeta}$. The Laplace transform $(\FM^o,\nabla^o)$ of $(M^o,\nabla^o)$ is the $\Czeta$-module obtained as follows. The $\CC$-vector space $\FM^o$ is equal to $M^o$. The action of $\zeta$ is defined as that of $\partial_\la$, and the action of $\partial_\zeta$ is defined as that of~$-\la$. Classical results show that $\FM^o$ is a holonomic $\Czeta$-module with a regular singularity at $\zeta=0$ and an irregular singularity at $\zeta=\infty$.

Tensoring this Laplace transform $\FM^o$ with $\CC[\zeta,\zeta^{-1}]$ over $\CC[\zeta]$ and setting $z=\zeta^{-1}$, we obtain a meromorphic bundle with connection $(G^o,\nabla^o)$ on the affine line~$\CC_z$ with pole at $z=0$ and an irregular singularity there (and a regular singularity at $z=\infty$). In other words, $(G^o,\nabla^o)$ is the \emph{localized Laplace transform} of $(M^o,\nabla^o)$ with respect to the Laplace kernel $e^{\la/z}$.

The \emph{formal stationary phase formula} describes the formalized connection
\[
(\wh G^o,\wh\nabla^o):=\CC\lpr z\rpr\otimes_{\CC[z]}(G^o,\nabla^o).
\]
There exists a $\CC\lpr z\rpr$-basis of $\wh G^o$ in which the matrix of $\wh\nabla^o$ is block-diagonal, with blocks of the form
\begin{equation}\label{eq:staphase}
u_i^o\id _{r_i} \rd(1/z)+C_i\,\rd z/z,\quad i=1,\dots,n,
\end{equation}
where $C_i$ is a constant matrix. Correspondingly, there is a pair $(S_+^o,S_-^o)$ of Stokes matrices that enables one to recover, up to isomorphism, $(G^o,\nabla^o)$ from $(\wh G^o,\wh\nabla^o)$. A~topological computation of $(S_+^o,S_-^o)$ from a presentation of $j_*L^o$ is given in \hbox{\cite[Th.\,5.4]{D-H-M-S17}}, adapting the more analytic approach in \cite[Chap.\,XII]{Malgrange91}. We recall it here in the special case of the intermediate extension $j_*L^o$.

We fix an \emph{$u^o$-admissible} argument $\theta^o$ in~$\CC_\la$, in the sense that the closed real half-lines~$\ell_i$ with this direction starting from $u_i^o$ does not contain $u_j^o$ for $j\neq i$. We set $\ell^o=\bigcup_i\ell_i$. We denote by~$D_i$ the maximal open strip in that direction containing $u_i^o$ and no other~$u_j^o$. See Figure \ref{fig:1}.
\begin{figure}[htb]
\begin{center}
\begin{tikzpicture}[scale=.9,
	background rectangle/.style={draw=white,dashed,fill=white},
	show background rectangle]
\begin{scope}
\clip (-1.5,-1.2) rectangle (1.8, 1.5);
\coordinate (C1) at (-1,-.2);
\coordinate (C1a) at (-.5,.3);
\coordinate (C2) at (0,.5);
\coordinate (C2a) at (.5,.1);
\coordinate (C3) at (1,.3);
\draw (C1) \crosss;
\draw (C2) \crosss;
\filldraw[fill=black, draw=black] (C1a) circle (2pt) node[below]{$u_i^o$};
\draw (C3) \crosss node[below]{$u_j^o$};
\draw[thick] (C1a) -- node[right]{$\ell_i$}++(0,1.3);
\draw (C2a) \crosss;
\end{scope}
\end{tikzpicture}
\quad
\begin{tikzpicture}[scale=.9,
	background rectangle/.style={draw=white,dashed,fill=white},
	show background rectangle]
\begin{scope}
\clip (-1.5,-1.2) rectangle (1.8, 1.5);
\coordinate (C1) at (-1,-.2);
\coordinate (C1a) at (-.5,.3);
\coordinate (C2) at (0,.5);
\coordinate (C2a) at (.5,.1);
\coordinate (C3) at (1,.3);
\filldraw[fill=black!20, draw=black!20] (-1,-2) rectangle (0,2);
\filldraw[fill=white, draw=white] (C1a)++(-.02,0) rectangle (-.48,2);
\filldraw[fill=white, draw=white] (C1a)++(-.02,-.02) rectangle ++(.04,2);
\filldraw[fill=white, draw=white, ultra thin] (C1) circle (2pt);
\filldraw[fill=white, draw= white] (C1a) circle (2pt) node[below]{$u_i^o$};
\filldraw[fill=white, draw= white] (C2) circle (2pt);
\draw (C3) \crosss node[below]{$u_j^o$};
\draw (C2a) \crosss;
\draw (-.5,-.7) node {$D_i \smallsetminus\ell_i$};
\end{scope}
\end{tikzpicture}
\quad
\begin{tikzpicture}[scale=.9,
	background rectangle/.style={draw=white,dashed,fill=white},
	show background rectangle]
\begin{scope}
\clip (-1.5,-1.2) rectangle (1.8, 1.5);
\coordinate (C1) at (-1,-.2);
\coordinate (C1a) at (-.5,.3);
\coordinate (C2) at (0,.5);
\coordinate (C2a) at (.5,.1);
\coordinate (C3) at (1,.3);
\filldraw[fill=black!20, draw=black] (-2,-2) rectangle (2,2);
\filldraw[fill=white, draw= white] (C1)++(-.02,-.02) rectangle ++(.04,2);
\filldraw[fill=white, draw= white] (C1a)++(-.02,-.02) rectangle ++(.04,2);
\filldraw[fill=white, draw= white] (C2)++(-.02,-.02) rectangle ++(.04,2);
\filldraw[fill=white, draw= white] (C2a)++(-.02,-.02) rectangle ++(.04,2);
\filldraw[fill=white, draw= white] (C3)++(-.02,-.02) rectangle ++(.04,2);
\filldraw[fill=white, draw= white] (C1) circle (2pt);
\filldraw[fill=white, draw= white] (C1a) circle (2pt);
\filldraw[fill=white, draw= white] (C2) circle (2pt);
\filldraw[fill=white, draw= white] (C2a) circle (2pt);
\filldraw[fill=white, draw= white] (C3) circle (2pt);
\draw (.2,-.7) node {$\CC_\la \smallsetminus \ell^o$};
\end{scope}
\end{tikzpicture}
\end{center}
\caption{\label{fig:1}}
\end{figure}

To $j_*L^o$ and the choice of a general argument $\theta^o$ as above (that we omit in the notation) is attached a quiver consisting of finite-dimensional vector spaces and linear maps between them. The quiver takes the form $(\Psi^o,\Phi^o_{i=1,\dots,n},\rc_i,\rv_i)$ and, due to the special case we consider here, it is obtained as follows (\cf\cite[Lem.\,4.8]{D-H-M-S17}):
\begin{itemize}
\item
$\Psi^o=H^2_\rc(\CC_\la\moins\ell^o,L^o)$,
\item
$\Psi^o_i=H^1_\rc(\ell_i^*,L^o)$, with $\ell_i^*:=\ell_i\moins\{u_i^o\}$, and monodromy $\rT_{ii}$,
\item
$h_i:\Psi^o_i\isom\Psi^o$ is an isomorphism obtained through $(4.6)_c$ in \loccit, and $\rT_i=h_i\circ\rT_{ii}\circ h_i^{-1}:\Psi^o\to\Psi^o$ is the $i$-th monodromy on $\Psi^o$,
\item
$\Phi^o_i=\image(\id-\rT_{ii})\subset\Psi^o_i$, $\rc_{ii}= (\id-\rT_{ii}) :\Psi^o_i\to\Phi^o_i$ and $\rv_{ii}:\Phi^o_i\hto\Psi^o_i$ is the natural inclusion,
\item
Last, $\rc_i=\rc_{ii}\circ h_i^{-1}$, $\rv_i=h_i\circ \rv_{ii}$.
\end{itemize}
We note that $\rv_i:\Phi^o_i\to\Psi^o$ is injective and its image is identified with $\image(\id-\rT_i)$. The quiver is thus isomorphic to the quiver
\[
(\Psi^o,\phi^o_{i=1,\dots,n},\rc_i,\rv_i),\quad \begin{cases}\phi^o_i=\image(\id-\rT_i),\\ \rc_i=(\id-\rT_i):\Psi^o\to\Psi^o,\\ \rv_i=\text{inclusion}:\phi^o_i\hto\Psi^o.
\end{cases}
\]

With this notation, \cite[Th.\,5.4]{D-H-M-S17} asserts (following \cite[Chap.\,XII]{Malgrange91}) that there exists a pair $(S_+^o,S_-^o)$ of Stokes matrices for $G^o$ which are decomposed into blocks $(i,j)$ with $i,j=1,\dots,n$, such that the non-diagonal blocks $(i,j)$ and $(j,i)$ read
\begin{itemize}
\item
$\rc_j\circ\rv_i$ and $0$ for $S_+^o$,
\item
$0$ and $-\rc_i\circ\rv_j$ for $S_-^o$.
\end{itemize}

\begin{corollaire}\label{cor:vanishingStokes}
With these assumptions, for $i\neq j\in\{1,\dots,n\}$, the above representative $(S_+^o,S_-^o)$ of Stokes matrices for $G^o$ has vanishing blocks $(i,j)$ and $(j,i)$ if and only if
\[
(\id-\rT_j)_{|\image(\id-\rT_i)}=0\quad\text{and}\quad (\id-\rT_i)_{|\image(\id-\rT_j)}=0.\eqno\qed
\]
\end{corollaire}

\section{Interpretation in terms of constancy of the vanishing cycle sheaf}\label{sec:interpretation}

If we vary $u^o\in\CC^n$ along a path $u(t)$ ($t\in[0,1]$) with the condition that $u_i(t)$ remain pairwise distinct, the local system~$L^o$ deforms in a unique way as a family of local systems~$L^t$ on $\CC_\la\moins\{u(t)\}$: this is obtained by an argument on the fundamental group (since the path is simply connected). Assume now that the limit point $u(1)$ has components which coincide, a property that the authors of \cite{C-D-G18} call \emph{coalescence}. The behaviour near such a point is better seen within a geometric setup.

Let us consider the space $\CC^n\times\CC_\la$ with coordinates $u_1,\dots,u_n,\la$ and the projection \hbox{$p:(u_1,\dots,u_n,\la)\mto(u_1,\dots,u_n)$}. Let us also consider the following hypersurfaces:
\begin{itemize}
\item
the hyperplanes $H_i=\{\la-u_i=0\}$ of $\CC^n\times\CC$ on the one hand,
\item
the hyperplanes $\Delta_{i,j}=\{u_i=u_j\}$ ($i\neq j$) of $\CC^n$ on the other hand, that are lifted to $\CC^n\times\CC$ as $\Delta_{i,j}\times\CC$.
\end{itemize}
The union of these are respectively denoted by $H$ and $\Delta$ (and $\Delta\times\CC$). We note that~$H$ is a divisor with normal crossings in $\CC^n\!\times\!\CC$ whose singular set projects bijectively~to~$\Delta$.

Let~$u^c$ be a fixed coalescing point, that is, a point on~$\Delta$. It defines a partition $\{1,\dots,n\}=\bigsqcup_{a=1}^mI^c_a$ such that $u^c_i=u^c_j$ if and only if~$i$ and~$j$ belong to the same subset $I^c_a$ and we denote by $u^c_a$ this common value. We~decompose correspondingly~$\CC^n$ as $\prod_{a=1}^m\CC^{I^c_a}$.

Let us fix an $u^c$-admissible argument $\theta^c$ and open subsets $D_a^c\subset\CC_\la$ ($a=1,\dots,m$) as in Figure \ref{fig:1}. The open subset $\DD_a^c$ of $\CC^{I^c_a}$ consisting of points whose coordinates belong to~$D_a^c$ is the product of the open strips $D_a^c$ in each coordinate plane, hence is homeomorphic to $\CC^{I^c_a}$. We then set $\DD^c\!=\!\prod_{a=1}^m\DD_a^c$, and we have \hbox{$\DD^c\!\moins\!\Delta\!=\!\prod_{a=1}^m(\DD_a^c\!\moins\!\Delta_a)$}. The disjoint union
\[
\coprod_{a=1}^m(\DD^c\times D_a^c)
\]
is an open neighbourhood of $(u^c,\{u^c\})$ in $\DD^c\times\CC_\la$ and
\[
H\cap(\DD^c\times D_a^c)=H_a\times\prod_{b\neq a}\DD^c_b,
\]
where $H_a\subset\DD^c_a\times D_a^c$ is defined by $\prod_{i\in I_a^c}(\la-u_i)=0$. We now call $H_i$ the intersection of $\{\la=u_i\}$ with $\DD^c\times\CC_\la$. For $i=1,\dots,n$, let $v_i(u,\la)=(\la-u_i):\DD^c\times\CC_\la\to\CC$ be the defining function of $H_i$. Then $H=\bigcup_iH_i$ is a normal crossing divisor in $\DD^c\times\CC_\la$. We also set $X=\DD^c\moins\Delta$ and denote by $j:(X\times\CC_\la)\moins H\hto X\times\CC_\la$ the open inclusion. Note that the intersection of $H$ with $X\times\CC_\la$ is smooth with components $H_i\cap(X\times\CC_\la)$, because the singular set of $H$ projects to $\Delta$. Also, the complement in~$H_i$ of $H_i\cap(X\times\CC_\la)$ is a normal crossing divisor in $H_i$.

Let $L'$ be a locally constant sheaf on $(\DD^c\times\CC_\la)\moins H$ and let $L$ denote its restriction to $(X\times\CC_\la)\moins H$. The \emph{nearby cycle complex} $\psi_{v_i}(L')$ is a complex of sheaves on $H_i$ equipped with an automorphism $T_i(\psi):\psi_{v_i}(L')\to\psi_{v_i}(L')$. By restricting over $X$, we obtain a locally constant sheaf $\psi_{v_i}(L)$ on $H_i\cap(X\times\CC_\la)$ equipped with the automorphism $T_i(\psi)$. For any $u^o\in X$, we denote by $L^o$ the restriction of $L$ to~\hbox{$u^o\times(\CC_\la\moins\{u^o\})$}.

The sheaf $j_*L$ on $X\times\CC_\la$ is a perverse sheaf (up to a shift) and the vanishing cycle sheaf $\phi_{v_i}(j_*L)$ is a locally constant sheaf on $H_i\cap(X\times\CC_\la)$: this is the sheaf\vspace*{-3pt}
\begin{equation}\label{eq:phi}
\image(\id-T_i(\psi)):\psi_{v_i}(L)\to\psi_{v_i}(L)
\end{equation}
equipped with the automorphism $T_i(\phi)$ induced by $T_i(\psi)$.

\begin{proposition}\label{prop:extjL}
With these notations, assume that for each $i=1,\dots,n$, the vanishing cycle local system $\phi_{v_i}(j_*L)$ on $H_i\cap(X\times\CC_\la)$ is constant. Then, for any $u^o\in X$, for any $a=1,\dots,m$ and for any pair $i\neq j\in I_a^c$, the Stokes matrices of~$G^o$ considered in Section \ref{sec:criterionvanishing} with respect to any argument~$\theta^o$ close enough to $\theta^c$ have their $(i,j)$- and $(j,i)$-blocks equal to zero.\end{proposition}

\begin{proof}
We first notice that, $u^o$ being fixed in $X$, if $\theta^c$ is not $u^o$-admissible, then any $\theta^o\neq\theta^c$ close enough to $\theta^c$ is $u^o$-admissible. The result does not depend on the choice of $\theta^o$.

The question is local at each $u^c_a$. Recall that $v_i\!=\!\la-u_i$. The open set $(\DD^c_a\times D_a^c)\moins H_a$ is homeomorphic to $(\CC^{I_a^c}\times\CC_\la)\moins\{\prod_{i\in I_a^c}v_i\!=\!0\}$ with coordinates $(v_{I_a^c},\la)$. We are thus considering a locally constant sheaf $L'$ on the complement of coordinate hyperplanes $v_i=0$ in $\CC^{I_a^c}\times\CC_\la$. Giving $L'$ is thus equivalent to giving a vector space $\mathrm{L}$ equipped with automorphisms $\rT_i$ ($i\!\in\!I_a^c$). We can choose $(\Psi^o,\rT_{i\in I_a^c})$ defined in Section \ref{sec:criterionvanishing} as such data.

With such a representation, the locally constant sheaf $\phi_{v_i}(j_*L)$ on $H_i\cap(X\times\CC_\la)$ is represented by the vector space with automorphisms $(\image(\id-\rT_i),\rT_{j\neq i})$. Furthermore, the automorphism $T_i(\phi)$ of $\phi_{v_i}(j_*L)$ corresponds to the automorphism induced by $\rT_i$. That $\phi_{v_i}(j_*L)$ is constant is equivalent to the property that each $\rT_{j\neq i}$ is the identity on $\image(\id-\rT_i)$. We conclude with Corollary \ref{cor:vanishingStokes}.
\end{proof}

\section{A short proof of a theorem of Cotti, Dubrovin and Guzzetti}\label{sec:shortproof}
We keep the setting and notation of Section \ref{sec:interpretation}: we fix a coalescing point $u^c\in\Delta$ and a point $u^o$ in $\DD^c\moins\Delta$.

We consider the trivial $\CC[z]$-module $F^o$ of rank~$n$ equipped with the connection~$\nabla^o$ having matrix
\begin{equation}\label{eq:Ao}
A^o=\Bigl(\frac{\Lambda^o}{z}+A_\infty^o\Bigr)\,\frac{\rd z}{z},\quad\text{where }\Lambda^o:=\diag(u_1^o,\dots,u_n^o).
\end{equation}
We assume that the only possible integral eigenvalues of $A_\infty^o$ are $\geq1$ and no diagonal entry of $A_\infty^o$ is an integer (this can be achieved by adding $c\id_n\rd z/z$ to~$A_\infty^o$ for a suitable $c\in\CC$, and Theorem \ref{th:CDG} is insensitive to this modification). The inverse Laplace (or Fourier) lattice $(E^o,\nabla^o)$ (\cf \cf \cite[Prop.\,V.2.10]{Bibi00}) is $F^o$ regarded as a $\CC[\la]$\nobreakdash-module, where $\la$ acts as $z^2\partial_z$. By the first assumption on $A_\infty^o$, it is free of rank $n$, with the same canonical basis as $F^o$, and the matrix of $\nabla^o$ in this basis is
\begin{equation}\label{eq:Bo}
B^o=(A_\infty^o-\id_n)(\la\id_n-\Lambda^o)^{-1}\rd\la,
\end{equation}
which takes the Fuchsian form \eqref{eq:Fuchs}.

\begin{lemme}\label{lem:point2}
The $\Cla$-submodule of $(V^o ,\nabla^o)$ generated by $E^o$ is the middle extension $(M^o ,\nabla^o)$ of $(V^o ,\nabla^o)$, whose localized Laplace transform $(G^o ,\nabla^o)$ is equal to $\wt G^o:=\CC[z,z^{-1}]\otimes_{\CC[z]}F^o$ with connection having matrix $A^o$.
\end{lemme}

\begin{proof}
Let us decompose $B^o$ as in \eqref{eq:Fuchs}. Then each matrix $B_i^o$ has rank one and a unique nonzero eigenvalue, which is the $i$-th diagonal entry of $A_\infty^o-\id$ and is non integral by the second assumption on $A_\infty^o$. Therefore, the matrix $B^o$ satisfies the assumption of Lemma \ref{lem:middleext}. This proves the first point.

Let $G^o$ be localized Laplace transform of $M^o$. Then
\[
E^o\subset M^o\implies F^o\subset G^o,\quad\text{hence }\wt G^o\subset G^o.
\]
In order to obtain equality, it is enough to show $\rk G^o =n$. By the stationary phase formula, this rank is equal to $\sum_{i=1}^n\phi_{\lambda-u_i^o}M^o$. Therefore, it is enough to show that, for each local monodromy~$\rT_i$ of $L^o =(V^o)^{\nabla^o}$ around $u_i^o$, we have $\rk(\id_n-\rT_i)=1$.

Since no two distinct eigenvalues of $B_i^o$ differ by an integer, the local monodromy~$\rT_i$ is conjugate to $\exp(-2\pi\sfi B_i^o)$, hence $\rT_i-\id$ has rank one, as desired.
\end{proof}

Since the diagonal terms of $\Lambda^o$ are pairwise distinct, we can write
\[
A_\infty^o=D_\infty^o+[\Lambda^o,R^o],
\]
where $D_\infty^o$ is the diagonal of $A_\infty^o$, and we can assume that the diagonal terms of~$R^o$ are zero.

A theorem of B.\,Malgrange \cite{Malgrange83cb, Malgrange86} asserts that there exists a universal integrable deformation of this system in the neighbourhood of $u^o$ (\cf also \cite[\S VI.3]{Bibi00}). In~particular (\cf \cite[\S VI.3.f]{Bibi00}), there exists a holomorphic matrix $R(u)$ on a simply connected neighbourhood $\nb(u^o)$ with zeros on the diagonal, such that the system
\begin{equation}\label{eq:familyBirkhoff}
\Bigl(\frac{\Lambda(u)}{z}+A_\infty(u)\Bigr)\frac{\rd z}{z},\quad A_\infty(u):=[\Lambda(u),R(u)]+D_\infty^o
\end{equation}
is integrable and $R(u^o)=R^o$. The diagonal part of $R(u)$ does not show up, so we can assume that the diagonal terms of $R(u)$ are zero. The integrable connection (on~the trivial bundle) has the matrix (\cf \cite[VI\,(3.12)]{Bibi00})
\begin{equation}\label{eq:univintegr}
-\rd(\Lambda(u)/z)+\bigl([\Lambda(u),R(u)]+D_\infty^o\bigr)\,\frac{\rd z}{z}-[\rd\Lambda(u),R(u)]
\end{equation}
and is a universal integrable deformation of its restriction at each point of the neighbourhood where it exists. Furthermore, on $\nb(u^o)$, there exists a $z$-formal base change which transforms \eqref{eq:univintegr} to the system
\begin{equation}\label{eq:univintegrform}
-\rd(\Lambda(u)/z)+D_\infty^o\,\frac{\rd z}{z}.
\end{equation}

\begin{theoreme}[\cite{C-D-G18}, \cite{Guzzetti21}]\label{th:CDG}
Assume that $A_\infty(u)$, defined on $\nb(u^o)$, extends holomorphically to $\DD^c$, and that, for any $a=1,\dots,m$ and any $i\neq j$ belonging to the same $I_a^c$, the $(i,j)$ and $(j,i)$ entries of $A_\infty(u)$ tend to zero when $u_i-u_j\to0$. Then the corresponding $(i,j)$ and $(j,i)$ entries of the Stokes matrices $(S^o_+,S^o_-)$ are zero.
\end{theoreme}

\begin{remarque}
We may restate the condition on $A_\infty(u)$ as the condition that $R(u)$ extends holomorphically to $\DD^c$.
\end{remarque}

\begin{proof}
We first note that the integrability property of \eqref{eq:univintegr} all over $\DD^c$ immediately results from that on $\nb(u^o)$. Indeed, integrability is equivalent to the property that~$R$ satisfies the following isomonodromy differential system on a neighbourhood of $u^o$:
\begin{equation}\label{eq:isomono}
\begin{cases}
\phantom{\rd}\rd [\Lambda,R]=-\Bigl[\bigl([\Lambda,R]+D_\infty^o\bigr),[\rd\Lambda,R]\Bigr],\\
\rd [\rd\Lambda,R]=[\rd\Lambda,R]\wedge[\rd\Lambda,R].
\end{cases}
\end{equation}
These are equalities on $\nb(u^o)$ between holomorphic matrices defined on $\DD^c$. Since~$\DD^c$ is connected, these equalities, hence the integrability property, hold all over $\DD^c$.

Let us set
\begin{equation}\label{eq:integrable}
\begin{split}
A(u,z)&=\Big(\frac{\Lambda(u)}{z}+[\Lambda(u),R(u)]+D_\infty^o\Big)\frac{\rd z}{z},\\
\Omega&=-\frac{\rd\Lambda(u)}{z}-[\rd\Lambda(u),R(u)].
\end{split}
\end{equation}

As explained in Lemma \ref{lem:point2}, the meromorphic bundle with connection $(G^o,\nabla^o)$ associated with the differential system of matrix $A^o$ given by \eqref{eq:Ao} is the localized Laplace transform of the middle extension $(M^o,\nabla^o)$ of the meromorphic bundle with connection $(V^o,\nabla^o)$ defined by the matrix $B^o$ given by \eqref{eq:Bo}. The existence of $A_\infty(u)$ on~$\DD^c$ together with integrability of \eqref{eq:integrable} implies the existence of a meromorphic bundle with integrable connection $(G,\nabla)$ on $\DD^c\times\CC$ restricting to $(G^o,\nabla^o)$ at $u^o$.

Let us consider the formalized connection along $z=0$:
\[
(\wh G,\wh\nabla):=\cO_{\DD^c}\lpr z\rpr\otimes_{\cO_{\DD^c}[z]}(G,\nabla).
\]
Then \cite[Prop.\,19.3]{C-D-G18} (see a reminder in Appendix \ref{app:A}) extends the formal decomposition \eqref{eq:univintegrform} for $(G,\nabla)$ all over $\DD^c$:
\begin{equation}\label{eq:wG}
(\wh G,\wh\nabla)\simeq\bigoplus_{i=1}^n\bigl(\cO_{\DD^c}\lpr z\rpr,\rd+\rd(u_i/z)+d_{\infty,i}^o\rd z/z\bigr),
\end{equation}
where $d_{\infty,i}^o$ is the $i$-th diagonal entry of the diagonal matrix $D_\infty^o$. The rank-one $\cO_{\DD^c}\lpr z\rpr$-module with connection $(\cO_{\DD^c}\lpr z\rpr,\rd+d_{\infty,i}^o\rd z/z)$ has regular singularity along $\{z=0\}$. It is uniquely determined as such by the data of the pair $(L_i,T_i)$, where $L_i=\CC_{\DD^c}$ is the \emph{constant} local system of rank one on~$\DD^c$, and $T_i$ is the automorphism of $L_i$ induced by multiplication by $\exp(-\twopii d_{\infty,i}^o)$.

Our aim is to apply Proposition \ref{prop:extjL} to a suitable local system $L$. Before doing so, we construct the local system $L'$ on $(\DD^c\times\CC_\la)\moins H$ by means of the standard result provided by Lemma \ref{lem:std}. We then define $L$ as the restriction of~$L'$ to $(X\times\CC_\lambda)\moins H$, and we prove that the local system $\phi_{v_i}(L)$ as considered in Section~\ref{sec:interpretation} is constant by identifying $\phi_{v_i}(L)$ with the restriction of $L_i$ to $X=\DD^c\moins\Delta$.

\begin{lemme}\label{lem:std}
Let $U$ be a simply connected complex manifold and let $D$ be a disc of some positive radius, centered at the origin in $\CC$. Let $u^o\in U$. The restriction functor at~$u^o$ induces an equivalence between the category of regular holonomic $\cD_{U\times D}$-modules with characteristic variety contained in \hbox{$T^*_{U\times D}(U\times D)\cup T^*_{U\times\{0\}}(U\times D)$} and the category of regular holonomic $\cD_{\{u^o\}\times D}$-modules with singularity at the origin only. A quasi-inverse functor is given by the pullback by the projection $U\times D\to \{u^o\}\times D$.\qed
\end{lemme}

Furthermore, we implicitly refer to \cite[App.\,A]{D-S02a} for passing from analytic to partially algebraic $\cD$-modules.

This being understood, we conclude, by taking $U=\DD^c$, that $\FM^o$ extends in a unique way as a holonomic $\Dzeta$-module $N$ with regular singularities along $\zeta=0$ and which satisfies (by setting $z=\zeta^{-1}$)
\begin{align*}
\cO_{\DD^c}[\zeta,\zeta^{-1}]\otimes_{\cO_{\DD^c}[\zeta]}N&=(G,\nabla),\\
N_{|u^o\times\CC}&=\FM^o.
\end{align*}
The inverse partial Laplace transform $M$ of $N$ is a holonomic $\cD_{\DD^c}[\la]\langle\partial_\la\rangle$-module, and we write $N=\FM$ (\cf Appendix~\ref{app:B}).

\begin{lemme}
The $\cD_{\DD^c}[\la]\langle\partial_\la\rangle$-module $M$ is smooth away from $H$ and defines there a locally constant sheaf $L'$.
\end{lemme}

\begin{proof}
We can regard $(G,\nabla)$ as a holonomic $\Dzeta$-module on which $\zeta$ acts in an invertible way. Let $\wt M$ denote its inverse Laplace transform. The cokernel of the inclusion $N\hto G$ is a holonomic $\Dzeta$-module supported on $\zeta=\nobreak0$. \hbox{Furthermore}, since the characteristic variety of $N$ and $G$ is as described in Lemma~\ref{lem:std}, that of $G/N$ is contained in $T^*_{\DD^c\times\{0\}}(\DD^c\times\CC_\zeta)$. As a consequence, the cokernel of $M\hto\wt M$ is a holonomic $\cD_{\DD^c}[\la]\langle\partial_\la\rangle$-module isomorphic to some power of $(\cO_{\DD^c}[\la],\rd)$. It~is thus enough to prove the lemma for $\wt M$. Recall also (\cf Appendix~\ref{app:B}) that $\wt M$ is nothing but~$G$ as a $\cD_{\DD^c}$-module, hence as an $\cO_{\DD^c}$. In particular, it is $\cO_{\DD^c}$\nobreakdash-flat and the restriction functor to any $u\in\DD^c$ only produces one cohomology $\Cla$\nobreakdash-mod\-ule $\wt M^u$, which is the inverse Laplace transform of the restriction $G^u$ to $u\times\CC^\la$. The formal stationary phase formula \eqref{eq:staphase} for $u$ fixed, together with the restriction of \eqref{eq:wG} at this value of $u$, implies that the singular set of $\wt M^u$ is equal to the finite set $\{u\}$. Since $u$ was arbitrary, this concludes the proof.
\end{proof}

From now on, we restrict to $X=\DD^c\moins\Delta$. We have $M_{|p^{-1}(u^o)}=M^o$ and~$M$ is regular holonomic, with poles on the \emph{smooth} hypersurface $H\subset X\times\CC_\la$. By~the uniqueness property of Lemma \ref{lem:std} applied to any simply connected open subset of~$X$ containing $u^o$, we conclude that $M$ is the middle extension of the meromorphic flat bundle $(V,\nabla):=\cO_X[\la,((\la-u_i)^{-1})_{i=1,\dots,n}]$ along $H$. Denoting by~$L$ the local system of horizontal sections of $(V,\nabla)$, $M$ is the regular holonomic $\cD_X[\la]\langle\partial_\la\rangle$-module associated to the (perverse) sheaf $j_*L$ via the Riemann-Hilbert correspondence.

The formalized connection $(\wh G,\wh\nabla)$ along $z=0$ is directly obtained from $(M,\nabla)$ by an operation called \emph{formal partial microlocalization} \hbox{\cite[Prop.\,1.18]{D-S02a}} (the non-characteristic assumption (NC) in \loccit\ is obviously satisfied here). We write $(\wh G,\wh\nabla)=p_*(M,\nabla)^\mu$, with the identification $z=\partial_\la^{-1}$ and $\la=z^2\partial_z$. By the standard identification recalled in Appendix \ref{app:B} (\cf Remark \ref{rem:B}), the vanishing cycle sheaf $\phi_{v_i}(j_*L)$ is a rank-one local system on $H_i\cap(X\times\CC_\la)$, identified with~$L_{i|X}$.

Since $L_i$ is constant on $\DD^c$, hence on $X$, we conclude the proof of Theorem \ref{th:CDG} by applying the criterion of Proposition \ref{prop:extjL}.
\end{proof}

\appendix
\section{A reminder of \texorpdfstring{\cite[Prop.\,19.3]{C-D-G18}}{CDG}}\label{app:A}

Let us recall the statement and proof of \cite[Prop.\,19.3]{C-D-G18} for the sake of completeness. We use the following notation. Given a square matrix $A$, we denote by~$A'$ the matrix formed of its diagonal terms, all off-diagonal terms being zero, and set $A''=A-A'$.

\begin{proposition}[{\cite[Prop.\,19.3]{C-D-G18}}]\label{prop:deformation}
In the setting of \eqref{eq:integrable}, there exists a unique $z$-formal base change
\[
\wh P(u,z)=\sum_{j\geq0}(-1)^jP_j(u)z^j,\quad P_0(u)\equiv\id,\ P_1''(u)=R(u),
\]
with $P_j(u)$ holomorphic on~$\DD^c$, such that
\begin{starequation}\label{eq:Fhat}
\begin{split}
\wh P^{-1}A(u,z)\wh P+\wh P^{-1}\partial_{z}\wh P\rd z&=\Big(\frac{\Lambda(u)}{z}+D_\infty^o\Big)\frac{\rd z}{z},\\
\wh P^{-1}\Omega\wh P+\wh P^{-1}\rd\wh P&=-\frac{\rd\Lambda(u)}{z}.
\end{split}
\end{starequation}%
\end{proposition}

The existence on $\nb(u^o)$ of $\wh P$ satisfying \eqref{eq:Fhat} and with given~$P_1''$ is standard. Let us first prove that $\wh P$ is unique on $\nb(u^o)$ and given by recursive formulas starting from~$P_1''$. For that purpose, we will only need to know that $\wh P$ satisfies the first line of \eqref{eq:Fhat} on~$\nb(u^o)$.

Setting $A_\infty''=[\Lambda,P_1'']$, the $P_j$'s are solutions of the following recursive equations ($P_0=\id$):
\begin{equation}\label{eq:recursive}
[\Lambda,P_{j+1}]=jP_j+A_\infty''P_j+[D_\infty^o,P_j].
\end{equation}
We have
\[
[\Lambda,P_{j+1}]=[\Lambda,P''_{j+1}]=[\Lambda,P''_{j+1}]'',
\]
and similarly $[D_\infty^o,P_j]=[D_\infty^o,P''_j]''$. We also have $(A_\infty''P_j)'=(A_\infty''P''_j)'$. Since $\Lambda$ is regular on $\nb(u^o)$, $P''_{j+1}$ ($j\geq1$) is uniquely determined by $P_j$ and $A_\infty''$, hence~$P_1''$, by the relation
\[
[\Lambda,P''_{j+1}]=jP''_j+(A_\infty''P_j)''+[D_\infty^o,P''_j].
\]
On the other hand, the diagonal part of \eqref{eq:recursive} for $j+1$ ($j\geq0$) reads
\begin{equation}\label{eq:Fjdelta}
P_{j+1}'=-(A_\infty''P''_{j+1})'/(j+1),
\end{equation}
hence $P_{j+1}'$ is uniquely determined by $P''_{j+1}$ and $P_1''$ (through $A_\infty''$), thus by $P_j$ and~$P_1''$. For $P_1'$ we obtain
\begin{equation}\label{eq:Fprime1}
P_1'=-(A_\infty''R)'.
\end{equation}

Let us now show that the $P_j$'s extend to $\DD^c$. We will use that $\wh P$ also satisfies the second line of \eqref{eq:Fhat} on $\nb(u^o)$. The system of this second line on $\DD^c$ can be written~as
\[
[\wh P,\rd\Lambda]=-z\bigl(\rd\wh P+[\rd\Lambda,R]\wh P\bigr).
\]
Denoting by $E_i$ the matrix having $(E_i)_{ab}=1$ if $a=b=i$ and $0$ otherwise, so that $\rd\Lambda=\sum_iE_i\,\rd u_i$, these equations are written
\[
[P_{j+1},E_i]=-\partial P_j/\partial{u_i}-[E_i,R]P_j.
\]
By \eqref{eq:Fprime1}, $P_1'$ (hence $P_1$) extends holomorphically to $\DD^c$. For $j\geq1$, let us assume that $P_1,\dots,P_j$ extend holomorphically to $\DD^c$. Then so does $[P_{j+1},E_i]$ for each $i$, therefore $P''_{j+1}$ also, hence $P_{j+1}'$ also, according to \eqref{eq:Fjdelta}.\qed

\section{A reminder on partial Laplace transformation, \texorpdfstring{\\partial}{partial} microlocalization and vanishing cycles}\label{app:B}
We recall here, following the method of \cite[\S1]{D-S02a}, how the formal stationary phase formula with parameters can be expressed in terms of vanishing cycles, by means of formal partial microlocalization. We take up the setting and notation of Section~\ref{sec:interpretation} with $X=\DD^c\moins\Delta$.

Let $(M,\nabla)$ be any regular holonomic $\cD_{\DD^c}[\la]\langle\partial_\la\rangle$-module. Recall that the \emph{partial Laplace transform} $\FM$ of $M$ is the $\cD_{\DD^c}[\zeta]\langle\partial_\zeta\rangle$-module which is equal to $M$ as a $\cD_{\DD^c}$\nobreakdash-module and on which $\zeta$ acts as $\partial_\la$ and $\partial_\zeta$ as $-\la$. It is holonomic with singularities along $\DD^c\times\{\zeta=0\}$. From $\FM$ we recover $M$ by \emph{inverse partial Laplace transformation}. Let us set $G=\cO_{\DD^c}[\zeta,\zeta^{-1}]\otimes_{\Ozeta}\FM$ and $z=\zeta^{-1}$. Then $G$ is a locally free $\cO_{\DD^c}[z,z^{-1}]$-module of finite rank with connection $\nabla$ having singularity along $z=0$.

Let $(\wh G,\wh\nabla)$ denote the formalization $\cO_{\DD^c}\lpr z\rpr\otimes_{\cO_{\DD^c}[z,z^{-1}]}(G,\nabla)$. We now implicitly restrict $M$ to $X\times\CC_\la$ and $G$ to $X\times\CC_z$. We will recall the proof of the following ``formal stationary phase formula with parameters''.

\begin{proposition}\label{prop:B}
Assume that the characteristic variety of $M$ is contained in the union of the zero section of $T^*(X\times\CC_\la)$ and the conormal bundles of the smooth hypersurfaces $H_i=\{v_i=0\}$ ($v_i:=\lambda-u_i$). Then $(\wh G,\wh\nabla)$ decomposes~as
\begin{starequation}\label{eq:B}
\bigoplus_{i=1}^n(R_i[z^{-1}],\nabla_i+\rd(u_i/z)),
\end{starequation}%
where $(R_i,\nabla_i)$ is a locally free $\cO_X\lcr z\rcr$-module of finite rank with logarithmic connection~$\nabla_i$, and the locally free $\cO_X$-module $R_i/zR_i$, equipped with the residual connection~$\nabla_i^\res$, is~isomorphic to $(\phi_{v_i}M,\nabla^\res)$.
\end{proposition}

Let $V^\cbbullet M$ denote the Kashiwara-Malgrange filtration of~$M$ along~$H$, that we consider here as indexed by integers. Due to our assumption on the singularity of $M$, each $V^kM$ is a coherent $\cO_X[\la]$-module. The connection $\nabla$ acts on each $V^kM$ with logarithmic poles along~$H$ and its residue has eigenvalues with real parts in $[k,k+1)$. The vanishing cycle module $\phi_HM$ is by definition the quotient $\gr_V^{-1}M=V^{-1}M/V^0M$. It is a locally free $\cO_H$\nobreakdash-module equipped with the residual integrable connection that we denote by $\nabla^\res$. Since $\gr_V^{-1}M$ is supported on $H=\coprod_iH_i$, $p_*\gr_V^{-1}M$ decomposes as $\bigoplus_i(\phi_{v_i}M,\nabla^\res)$. Since the projection $p:X\times\CC_\la\to X$ induces an isomorphism $H_i\simeq X$, we can regard each $(\phi_{v_i}M,\nabla^\res)$ as a locally free $\cO_X$-module with integrable connection. This explains why there can be an identification between $(\phi_{v_i}M,\nabla^\res)$ and $(R_i/zR_i,\nabla_i^\res)$.

\begin{remarque}\label{rem:B}
If $M$ is the regular holonomic $\cD_X[\la]\langle\partial_\la\rangle$-module associated to $j_*L$ by the Riemann-Hilbert correspondence, then, due to the compatibility between taking vanishing cycles and the Riemann-Hilbert functor (\cf\eg\cite{Bibi87,M-M04}), $(\phi_{v_i}M,\nabla^\res)$ corresponds to $\phi_{v_i}(j_*L)$, and Proposition \ref{prop:B} provides the claim in the last part of the proof of Theorem~\ref{th:CDG}.
\end{remarque}

\begin{proof}[Proof of Proposition \ref{prop:B}]
Let $M[\partial_\la^{-1}]$ be the regular holonomic $\cD_X[\la]\langle\partial_\la\rangle$-module whose partial Laplace transform is exactly $G$ (\ie $M[\partial_\la^{-1}]$ is the inverse Laplace transform of $G$) and let $V^\cbbullet(M[\partial_\la^{-1}])$ its Kashiwara-Malgrange filtration along~$H$. We~claim that, for any $k\in\ZZ$, $V^k(M[\partial_\la^{-1}])=\partial_\la^kV^{-1}(M[\partial_\la^{-1}])$ and that the natural morphism $M\to M[\partial_\la^{-1}]$ induces an isomorphism $(\phi_{v_i}M,\nabla^\res)\isom(\phi_{v_i}(M[\partial_\la^{-1}]),\nabla^\res)$ for each $i=1,\dots,n$. Due to the uniqueness of the Kashiwara-Malgrange filtration, we can work on an analytic neighbourhood $U_i$ of $H_i$ in $X\times\CC_\la$, that the flat functor $\cO_{U_i}\otimes_{\cO_X[\la]}({\scriptstyle\bullet})$ preserves the Kashiwara-Malgrange filtration. Then the claim follows from the standard properties of this filtration (\cf\eg\cite{Bibi87,M-M04}). As a consequence, we can assume that $M=M[\partial_\la^{-1}]$, and $\bigoplus_i(\phi_{v_i}M,\nabla^\res)\simeq(\gr^{-1}_VM,\nabla)$ has rank equal to the rank of $M$ which is also equal to the rank of $G$.

By \cite[Prop.\,1.18]{D-S02a}, $(\wh G,\wh\nabla)$ is identified with the partial microlocalized module $p_*M^\mu$, as defined in \loccit, and the microlocalized lattice $p_*(V^{-1}_M)^\mu$ is a coherent $\cO_X\lcr z\rcr$-module which is the direct sum of the components $p_*(V^{-1}M_{|U_i})^\mu$ for $i=1,\dots,n$, since $(V^{-1}M_{|U_i})^\mu$ is supported on $H_i$. By definition of the $V$-filtration, $V^{-1}M_{|U_i}$ is acted on by $\partial_\la\cdot(\la-u_i)$, $\partial_\la+\partial_{u_i}$ and $\partial_{u_j}$ for $j\neq i$. This is seen by changing the variables $(u_i,u_{j\neq i},\la)\mto(v_i=\la-u_i,u_{j\neq i},\la)$.

As for the corresponding actions on the partial Laplace transform, the action of~$z$ on $p_*M^\mu$ is that induced by $\partial_\la^{-1}$ and the action of $\partial_z$ by $\la^2\partial_\la$. Then $p_*(V^{-1}M_{|U_i})^\mu$ is acted on by $z\partial_z-u_i/z$, $z^{-1}+\partial_{u_i}$ and $\partial_{u_j}$ for $j\neq i$, operators which also read
\[
e^{-u_i/z}\cdot z\partial_z\cdot e^{u_i/z},\quad e^{-u_i/z}\cdot \partial_{u_i}\cdot e^{u_i/z},\quad e^{-u_i/z}\cdot \partial_{u_j}\cdot e^{u_i/z}\quad\text{for $j\neq i$}.
\]
In other words, if we twist the connection by $e^{-u_i/z}$, $p_*(V^{-1}M_{|U_i})^\mu$ becomes a cohe\-rent $\cO_X\lcr z\rcr$-module with logarithmic connection $(R_i,\nabla_i)$ having poles along $z=0$ only. As a consequence, $R_i$ is $\cO_X\lcr z\rcr$-locally free of finite rank. By the computation of the rank aforementioned, we find $\sum_i\rk R_i=\rk G$, and \cite[Prop.\,1.20]{D-S02a} applies, giving the decomposition \eqref{eq:B}.

Last, the quotient $R_i/zR_i$ with its residual connection $\nabla_i^\res$ is identified with $p_*(V^{-1}M_{|U_i}/\partial_\la^{-1}V^{-1}M_{|U_i})^\mu$. Because $V^{-1}M_{|U_i}/\partial_\la^{-1}V^{-1}M_{|U_i}=\phi_{v_i}M$ is supported on $H_i$, microlocalization does not change it, and we thus identify $(R_i/zR_i,\nabla_i^\res)$ with $(\phi_{v_i}M,\nabla^\res)$.
\end{proof}

\backmatter
\providecommand{\SortNoop}[1]{}\providecommand{\sortnoop}[1]{}\providecommand{\eprint}[1]{\href{http://arxiv.org/abs/#1}{\texttt{arXiv\string:\allowbreak#1}}}\providecommand{\hal}[1]{\href{https://hal.archives-ouvertes.fr/hal-#1}{\texttt{hal-#1}}}\providecommand{\tel}[1]{\href{https://hal.archives-ouvertes.fr/tel-#1}{\texttt{tel-#1}}}\providecommand{\doi}[1]{\href{http://dx.doi.org/#1}{\texttt{doi\string:\allowbreak#1}}}
\providecommand{\bysame}{\leavevmode ---\ }
\providecommand{\og}{``}
\providecommand{\fg}{''}
\providecommand{\smfandname}{\&}
\providecommand{\smfedsname}{\'eds.}
\providecommand{\smfedname}{\'ed.}
\providecommand{\smfmastersthesisname}{M\'emoire}
\providecommand{\smfphdthesisname}{Th\`ese}


\begin{thebibliography}{DHMS20}

\bibitem[CDG19]{C-D-G18}
{\scshape G.~Cotti, B.~Dubrovin {\normalfont \smfandname} D.~Guzzetti} -- {\og
  Isomonodromy deformations at an irregular singularity with coalescing
  eigenvalues\fg}, \emph{Duke Math.~J.} \textbf{168} (2019), no.~6,
  p.~967--1108.

\bibitem[DHMS20]{D-H-M-S17}
{\scshape A.~D'Agnolo, M.~Hien, G.~Morando {\normalfont \smfandname} C.~Sabbah}
  -- {\og {Topological computation of some Stokes phenomena on the affine
  line}\fg}, \emph{Ann. Inst. Fourier (Grenoble)} \textbf{70} (2020), no.~2,
  p.~739--808.

\bibitem[DS03]{D-S02a}
{\scshape A.~Douai {\normalfont \smfandname} C.~Sabbah} -- {\og {Gauss-Manin
  systems, Brieskorn lattices and Frobenius structures~(I)}\fg}, \emph{Ann.
  Inst. Fourier (Grenoble)} \textbf{53} (2003), no.~4, p.~1055--1116.

\bibitem[Guz21]{Guzzetti21}
{\scshape D.~Guzzetti} -- {\og Isomonodromic {L}aplace transform with
  coalescing eigenvalues and confluence of {F}uchsian singularities\fg},
  \emph{Lett. Math. Phys.} \textbf{111} (2021), no.~3, Paper No.\,80, 70\,p.

\bibitem[MM04]{M-M04}
{\scshape {\relax Ph}.~Maisonobe {\normalfont \smfandname} Z.~Mebkhout} -- {\og
  {Le th{\'e}orème de comparaison pour les cycles {\'e}vanescents}\fg}, in
  \emph{{{\'E}l{\'e}ments de la th{\'e}orie des systèmes diff{\'e}rentiels
  g{\'e}om{\'e}triques}}, S{\'e}minaires \& Congrès, vol.~8, Soci{\'e}t{\'e}
  Math{\'e}matique de France, Paris, 2004, p.~311--389.

\bibitem[Mal83]{Malgrange83cb}
{\scshape B.~Malgrange} -- {\og {D{\'e}formations de systèmes
  diff{\'e}rentiels et microdiff{\'e}rentiels}\fg}, in \emph{{S{\'e}minaire
  E.N.S. Math{\'e}matique et Physique}} (L.~Boutet~{de Monvel}, A.~Douady
  {\normalfont \smfandname} J.-L. Verdier, \smfedsname), Progress in Math.,
  vol.~37, Birkh{\"a}user, Basel, Boston, 1983, p.~351--379.

\bibitem[Mal86]{Malgrange86}
\bysame , {\og {Deformations of differential systems, II}\fg},
  \emph{J.~Ramanujan Math. Soc.} \textbf{1} (1986), p.~3--15.

\bibitem[Mal91]{Malgrange91}
\bysame , \emph{{\'E}quations diff{\'e}rentielles {\`a} coefficients
  polynomiaux}, Progress in Math., vol.~96, Birkh{\"a}user, Basel, Boston,
  1991.

\bibitem[Sab87]{Bibi87}
{\scshape C.~Sabbah} -- {\og {{$\mathcal{D}$}-modules et cycles {\'e}vanescents
  (d'après B.~Malgrange et M.~Kashiwara)}\fg}, in \emph{G{\'e}om{\'e}trie
  alg{\'e}brique et applications (La Rábida, 1984)}, vol. III, Hermann, Paris,
  1987, p.~53--98.

\bibitem[Sab02]{Bibi00}
\bysame , \emph{{D{\'e}formations isomonodromiques et vari{\'e}t{\'e}s de
  Frobenius}}, Savoirs Actuels, CNRS~{\'E}ditions \& EDP~Sciences, Paris, 2002,
  English Transl.: Universitext, Springer \& EDP~Sciences, 2007.

\end{thebibliography}
\end{document}